\documentclass{llncs}

\usepackage{latexsym}
\usepackage{amssymb}
\usepackage{amsmath}
\usepackage{mathtools}
\usepackage{hyperref}
\usepackage{graphicx}
\usepackage{colortbl}

\usepackage[normalem]{ulem}
\usepackage{multicol}

\DeclareMathOperator{\boxi}{box} 

\DeclareMathOperator{\KG}{K}

\def\Ascr{\mathcal{A}}
\def\Bscr{\mathcal{B}}
\def\Cscr{\mathcal{C}}

\def\Escr{\mathcal{E}}
\def\Fscr{\mathcal{F}}

\def\Iscr{\mathcal{I}}

\def\Oscr{\mathcal{O}}

\usepackage{marginnote}

\usepackage[textwidth=2cm]{todonotes}

\title{\texorpdfstring{On the Boxicity of  Line Graphs \\ and of Their Complements\thanks{A shorter account of a part of the results of this paper appeared in the Proceedings of the 31st International Symposium on Graph Drawing and Network Visualization \cite{GD23_Caoduro}.
This article, besides containing full proofs of theorems of the announced theorems, presents new results on the boxicity of line graphs.}}{On the Boxicity of Line Graphs and of Their Complements}}

\author{Marco Caoduro\inst{1} \and Andr\'as Seb\H{o}\inst{2}}

\institute{Sauder School of Business, University of British Columbia, Vancouver, Canada \email{marco.caoduro@ubc.ca} \and CNRS, Laboratoire G-SCOP, Univ.~Grenoble Alpes, Grenoble, France \email{andras.sebo@cnrs.fr}}

\begin{document}

\maketitle

\begin{abstract}
 
    The boxicity of a graph is the smallest dimension $d$ allowing a representation of it as the intersection graph of a set of $d$-dimensional axis-parallel boxes. We present a simple general approach to determining the boxicity of a graph based on studying its ``interval-order subgraphs.''
    
    The power of the method is first tested on the boxicity of some popular graphs that have resisted previous attempts:  the boxicity of the Petersen graph is $3$,  and more generally, that of the Kneser-graphs $\KG(n,2)$ is $n-2$ if $n\ge 5$, confirming a conjecture of Caoduro and Lichev [Discrete Mathematics, Vol. 346, 5, 2023]. 

    As every line graph is an induced subgraph of the complement of $\KG(n,2)$, the developed tools show furthermore that line graphs have only a polynomial number of edge-maximal interval-order subgraphs. This opens the way to polynomial-time algorithms for problems that are in general $NP$-hard: for the existence and optimization of interval-order subgraphs of line graphs, or of interval completions and the boxicity of their complement, if the boxicity is bounded.  
    We finally extend our approach to upper and lower bounding the boxicity of line graphs.
   
\end{abstract}


\section{Introduction}

Boxicity is a graph parameter introduced by Roberts~\cite{1969_Roberts} in 1969.
The \emph{boxicity} of a graph $G=(V,E)$, denoted by $\boxi(G)$, is the minimum dimension $d$ such that $G$ is the intersection graph of a family of axis-parallel boxes in $\mathbb R^d$.
The \emph{intersection graph} of a finite set $\Fscr$ is the graph $G(\Fscr)$ having vertex-set $\{v_A: A \in \Fscr\}$ and edge-set $\{v_A v_B : A,B \in \Fscr, \ A \neq B, \textrm{ and } A\cap B \neq \emptyset\}$.
A graph $G$ has $\boxi(G) = 0$ if and only if $G$ is a complete graph,  $\boxi(G) \leq 1$ if and only if it is an {\em interval graph}, and $\boxi(G) \le k$ if it is the intersection of $k$ interval graphs~\cite{1983_Cozzens}.
Complements of interval graphs are called here {\em interval-order graphs}.

Interval graphs, and therefore also interval-order graphs, can be recognized in linear time \cite{1976_Booth} while determining whether a graph has boxicity at most $2$ is $NP$-complete~\cite{1994_Kratochvil_NP}. 
In the language of parameterized complexity, the computation of boxicity is not in the class \emph{XP} of problems solvable in polynomial time when the parameter boxicity is bounded by a constant.

Boxicity has been a well-studied graph parameter: 
Roberts~\cite{1969_Roberts} proved that any graph $G$ on $n$ vertices has boxicity at most $\left \lfloor \frac{n}{2} \right \rfloor$. 
Esperet~\cite{2016_Esperet} showed that the boxicity of graphs with $m$ edges is $\mathcal{O}(\sqrt{m\log{m}})$.
Results on planar graphs or graphs of bounded genus can be found in 
Scheinerman~\cite{1984_Scheinerman}, Thomassen~\cite{1986_Thomassen}, Chandran \&  Sivadasan~\cite{2007_Chandran}, Esperet~\cite{2013_Esperet}, and Esperet \&  Joret~\cite{2017_Esperet}.
For additional results on the boxicity of other special graph classes, see, e.g., \cite{2011_Bhowmick,2011_Chandran_Bipartite,2011_Chandran,2018_Kamipebbu}.

In this paper, we study the boxicity of line graphs and their complements. 
Given a graph $G$, the \emph{line graph} of $G$, denoted by $L(G)$, is the intersection graph of $E(G)$.
The \emph{complement} of a graph $G$ is denoted by $\overline G$, while the \emph{complete graph} on $n$ vertices, or a vertex set $V$, is denoted by $K_n$ or $K_V$, respectively.
Line graphs and their complements are an extensively studied class of graphs with connections to fundamental topics in graph theory.
For instance,  stable sets in line graphs (equivalently, cliques in their complements) correspond to matchings, their chromatic number encodes edge-colouring, and they play a central role for claw-free graphs.
Furthermore,  complements of line graphs include the simplest Kneser graphs $\KG(n,2)$, for $n\geq 5$.
Indeed, $\KG(n,2)$ is isomorphic to $\overline{L(K_n)}$.

Given $k$ and $n$ two positive integers such that $n \geq 2k + 1$, the \emph{Kneser graph} $\KG(n,k)$ is the graph whose vertex set consists of all subsets of $[n] \coloneqq \{1,2,\dots,n\}$ of size $k$, where two vertices are adjacent if their corresponding $k$-sets are disjoint.
Kneser graphs stimulated deep and fruitful graph theory ``building bridges'' with other parts of mathematics.
For instance, Lov\'asz's proof~\cite{1978_Lovasz} of Kneser's conjecture~\cite{1956_Kneser} is the  source of the celebrated ``topological method.''
Several papers have focused on other properties of Kneser graphs; see, e.g., \cite{2000_Chen,1986_Frankl,2013_Harvey,1978_Lovasz}.

The study of the boxicity of Kneser-graphs $\KG(n,k)$ was initiated by Stehl\'{i}k as a question \cite{QuestionMatej}. 
Caoduro and Lichev~\cite{2023_Caoduro} established a general upper bound of $n-2$, a lower bound of $n - \frac{13k^2-11k+16}{2}$  for $n\geq 2k^3 - 2k^2 + 1$, and a lower bound of $n-3$ for $k=2$ that nearly matches the upper bound.
They also conjectured that $\boxi(\KG(n,2)) = n-2$.
Already for the case $n=5$ new ideas are required; to the best of our knowledge, the only declared solution for this particular graph finishes with a computer-based case-checking~\cite[Section 6]{2023_Caoduro}.
We establish the conjecture in its full generality.

\begin{theorem}\label{thm:box_kneser}
	The boxicity of the Kneser-graph $\KG(n,2)$ is $n-2$.
    In particular, the boxicity of the Petersen graph $\KG(5,2)$ is $3$.
\end{theorem}

The proof of Theorem~\ref{thm:box_kneser} starts with a well-known rephrasing of boxicity in terms of  ``interval completion,''
allowing a graph theory perspective. 
An \emph{interval completion} of $G$ is an interval graph containing $G$ as a subgraph.

This approach reveals that each interval completion of $\KG(n,2)$ is uniquely determined by the choice of at most $5$ indices in $[n]$. 
Thus, the number of interval completions is polynomial.
Although the list is already non-trivial for small Kneser graphs, there are only $4$ essentially different interval completions of $\KG(n,2)$. 

Given a graph $G$ on $n$ vertices, $\overline{L(G)}$ is isomorphic to an induced subgraph of $\overline{L(K_n)} = \KG(n,2)$.
Leveraging this, we obtain that also the complement of any line graph has only a polynomial number of interval completions:

\begin{lemma}\label{lem:all_minimal_completion}
    Let $G=(V, E)$ be a graph on $n$ vertices.
    Then $\overline{L(G)}$ has at most $n^5$ inclusion-wise minimal interval completions, all of which can be enumerated in $\Oscr(n^7)$~time. 
\end{lemma}

Determining the minimum number of edges needed for an interval completion of a given graph is known as the {\em Interval Graph Completion problem (IGC)}.
This problem is among the first $NP$-hard problems~\cite{1979_Garey}.
However, by Lemma~\ref{lem:all_minimal_completion}, for the complement of a line graph, we can generate all its minimal interval completions and select the one with the fewest edges.
We obtain the following result.

\begin{theorem}\label{thm:minimum_completion} 
IGC is polynomial-time solvable for complements of line graphs. 
\end{theorem}

Polynomial complexity follows for detecting bounded boxicity, placing the computation of the boxicity of complements of line graphs in the class $XP$. 

\begin{theorem}\label{thm:box_co_line}
Let $G=(V, E)$ be a graph on $n$ vertices and $k$ a positive integer. Then it can be decided in $\Oscr(n^{5k+4})$-time whether $\boxi( \overline{L(G)} ) \leq k$. 
\end{theorem}

The boxicity of line graphs has been previously studied in~\cite{2011_Chandran}, where it is shown that for a line graph $L(G)$ with maximum degree $\Delta$, $\boxi(L(G))$ is  $\mathcal{O}(\Delta \log \log \Delta)$. 
This result improves upon an earlier bound in~\cite{2010_Adiga}, which established that the boxicity is $\mathcal{O}(\Delta \log^2 \Delta)$ for (general) graphs with maximum degree $\Delta$. %

The techniques developed for finding interval-order subgraphs of a graph can also be applied to study the boxicity of line graphs.
However, unlike for complements of line graphs, the number of interval-order subgraphs of line graphs is not polynomial.
Nevertheless, our simplified approach for exhibiting interval order subgraphs allows us to establish new bounds for the boxicity of line graphs.
\begin{theorem}\label{thm:line_upper}
Let $L(G)$ be a line graph on $n$ vertices.
Then $\boxi(L(G)) \leq \lceil 5 \log n\rceil$.
\end{theorem}
Theorem~\ref{thm:line_upper} improves the bound of $\lfloor \frac{n}{2} \rfloor$ for a graph on $n$ vertices by Roberts~\cite{1969_Roberts} and refines the result presented in~\cite{2011_Chandran} for line graphs, specifically when the line graph has maximum degree $\Delta = o(\log(n))$.

It can be seen that the largest interval-order subgraph of $\overline{L(K_n)}$ contains one-third of the edges of $\overline{L(K_n)}$.
As the boxicity of a graph $G$ can be characterized as the minimum number of interval-order subgraphs of $\overline{G}$ required to cover $E(\overline{G})$ (Lemma~\ref{lem:Coz_Rob}), one could be tempted to conjecture that the boxicity of $L(K_n)$ is bounded by a constant.
However, we establish a lower bound for $\boxi(L(K_n))$ depending on a monotone increasing function of $n$.
Our proof relies on a lemma of Spencer, which has also been connected to the subject by Chandran, Mathew, and Sivadasan~\cite{2011_Chandran} in a different way.

\begin{theorem}\label{thm:lower_bound_box_line_graph}
For every $n \geq 5$, $\boxi(L(K_n)) \geq \lfloor \log_2 \log_2 (n - 1) \rfloor + 3$. %
\end{theorem} %

\medskip

The paper is structured as follows.
Section~\ref{sec:prelim} introduces the first,  classical properties that play a fundamental role in the study of boxicity. It also characterizes interval-order graphs using the orderings of their vertices.
Section~\ref{sec:interval_orders} focuses on the interval-order subgraphs of line graphs and shows that all interval-order subgraphs of the line graph of $K_n$ can be easily described using only five vertices of $K_n$.
Section~\ref{sec:petersen} presents the proofs of the main results concerning the complements of line graphs.
These can be applied to proving that the boxicity of the Petersen graph is $3$, a problem known before as a challenging open question; see, e.g.,~\cite{2014_Bruhn}.
More generally, it determines the boxicity of the Kneser graphs $\KG(n,2)$.
This latter proof can be generalized to solving the interval completion problem for complements of line graphs in polynomial time and detecting constant boxicity.
While it places boxicity within the complexity class $XP$ for complements of line graphs, it remains open whether this problem is solvable in polynomial time.
Section~\ref{sec:line_graphs} establishes new lower and upper bounds on the boxicity of line graphs.
In contrast to complements of line graphs, the problem of computing the boxicity of line graphs is not known to lie in the class $XP$.
Finally, Section~\ref{sec:conclusion} concludes with some applications, possible extensions, and open questions.

\section{Preliminaries}\label{sec:prelim}

Graphs are {\em simple} in this paper, that is, they do not have loops or parallel edges. 
Given a graph $G=(V, E)$ and a vertex $v \in V$, $\delta(v)$ denotes the edges incident to $v$ and $N(v)$ the vertices adjacent to $v$. The {\em degree} of a vertex $v$ is $d(v)\coloneqq|\delta(v)| = |N(v)|$. For $V' \subseteq V$, $G[V']$ is the \emph{induced subgraph} of $G$ with vertex-set $V'$,  and edge-set  $\{ uv \in E \ : \ u,v \in V'\}$.

\subsection{Defining the boxicity using interval graphs}\label{subsec:box_int}

An \emph{axis-parallel box} in $\mathbb{R}^d$ is a Cartesian product $I_1\times \dots \times I_d$ where each $I_i$ is a closed interval in the real line.
Axis-parallel boxes $B, B'$ intersect if and only if for every $i \in [d]$ the intervals $I_i$ and $I'_i$ intersect. 
This simple fact immediately implies the following lemma.
\begin{lemma}[Roberts~\cite{1969_Roberts}, 1969]\label{lem:Roberts}
	Let $G  = (V,E)$ be a graph. Then $\boxi(G) \leq k$ if and only if there are $k$ interval graphs $G_1, \ldots, G_k$ on the vertex set $V$ such that $E = \bigcap_{i=1}^k E(G_i)$.
\end{lemma}

The interval graphs in Lemma~\ref{lem:Roberts} are interval completions of $G$.
Therefore, it is useful to explore the interval completions of a graph for deducing bounds on its boxicity.
By de Morgan's law,   a graph $G$ is the intersection of interval completions of $G$ if and only if its complement is the union of interval order subgraphs of $\overline{G}$.
This motivates the following definition:

Let $G=(V, E)$ be a graph. We say that the family of edge-sets  $\Cscr = \{C_1, \ldots,$ $C_k\}$ $(C_i\subseteq E)$ is a \emph{interval-order cover}, or a \emph{$k$-interval-order cover} of $E$ (or $G$) if each $(V, C_i)$ is an interval-order graph and $E = \bigcup_{i=1}^k C_i$.
\begin{lemma}[Cozzens and Roberts~\cite{1983_Cozzens}, 1983]\label{lem:Coz_Rob}
	Let $G = (V,E)$ be a graph. Then $\boxi(G) \leq k$ if and only if $\overline G$ has a  $k$-interval-order cover.
\end{lemma}

Note that by deleting or adding an edge, the boxicity may increase.
However, for any graph $G = (V,E)$ and $v \in V$,
$$\boxi(G[V \setminus \{v\}]) \leq \boxi(G) \leq \boxi(G[V \setminus \{v\}]) + 1.$$
The second inequality can be strengthened, as shown in the following result.
\begin{lemma}[Roberts~\cite{1969_Roberts}, 1969]\label{lem:box_minus_uv}
	Let $G  = (V, E)$ be a graph. If $u,v \in V$ and $uv \notin E$, then $\boxi(G) \leq \boxi(G[V \setminus \{u,v\}]) + 1$.
\end{lemma}
We can easily obtain the following classical result by induction on the number of vertices of the graph.
\begin{theorem}[Roberts~\cite{1969_Roberts}, 1969]\label{thm:Roberts}
Let $G  = (V, E)$ be a graph on $n$ vertices. Then $\boxi(G) \leq \lfloor \frac{n}{2} \rfloor$.
\end{theorem}

\subsection{Defining interval-order graphs using orderings on their vertices}\label{subsec:ordering}
%
Given a family $\Iscr$ of $n$ closed intervals in $\mathbb{R}$,  order them in a non-decreasing order $\sigma= (I_1, \ldots, I_n)$ of their right endpoints, and consider the interval-order graph $G=(V, E)$, and  $V=\{v_1,\ldots, v_n\}$, where, for $i \in [n]$, $v_i$ corresponds to $I_i$, and two vertices are joined if and only if the corresponding intervals are disjoint (the complement of their intersection graph).
Orienting the edges of $G$ from the vertices of smaller  index towards those of larger index, we have:
\begin{equation}\label{eq:chain_property}    
	\textit{if $i > j$, then $N^+(v_i) \subseteq N^+(v_j)$,}
\end{equation}
where the \emph{out-neighbourhood} of a vertex $v_l \in V(G)$ is the set $N^+(v_l) \coloneqq \{v_m : v_lv_m  \in E(G),\ l < m\}$.

Clearly, the series of the sizes of $N^+(v_i)$ for $i \in [n]$ is (not necessarily strictly) monotone decreasing from $d(v_1)$ to $0$.
In addition, {\em in  the  chain  $N^+(V_{1})\supseteq\ldots \supseteq N^+(V_{n})$ only the first at most $D$ sets are not empty, where $D$ is the maximum degree of $G$.} Indeed, suppose  $N^+(V_{D+1}) \ne \emptyset$, and let $u\in N^+(V_{D+1})$. Then by (\ref{eq:chain_property}), $u$ is also in the neighbourhood of all previous vertices, that is,  $N_G(u)\supseteq \{v_1,\ldots, v_{D + 1}\}$, contradicting that $D$ is the maximum degree.

By necessity, one realizes that (\ref{eq:chain_property})  actually  {\em characterizes} interval-order graphs (the proof is immediate by induction).  It is explicitly stated in Olariu's paper~\cite{1991_Olariu}:

\begin{lemma}\label{lem:co_intervals} 
	A graph $G=(V, E)$ is an interval-order graph if and only if $V$ has an ordering  $(v_1,\ldots, v_n)$ so that (\ref{eq:chain_property}) holds. 
\end{lemma}

Inspired by this characterization, we can easily construct interval-order subgraphs of an arbitrary graph $G = (V,E)$ starting from permutations of $V$.
To show how, we introduce the following construction:

Let $G=(V,E)$ be a graph and $\sigma$ an ordering of $V$. We define the graph $G^\sigma=(V,E^\sigma)$ as follows: let $V_0\coloneqq V$, $V_i \coloneqq V_{i-1}\cap N_G (v_i)$ for $i \in [n]$, and $E^\sigma\coloneqq E_1 \cup \ldots \cup E_{n} \subseteq E$ where $E_i$ is the set of edges from $v_i$ to $V_i$; see Fig.~\ref{fig:example_G_sigma} for an example.
Note that the set $V_i$\,--\,that is eventually becoming $N^+(v_i)$  and this is the notation we will use for it\,--\,depends only on the graph $G$ and the  {\em $i$-prefix}  $\sigma_i\coloneqq(v_1, \ldots, v_i)$ of $\sigma$. 

If  $N^+(v_{i+1})=\emptyset$, then $E_j$ for $j> i$ does not add   any more edges to $E^\sigma$, so the {\em $i$-suffix} $(v_{i+1},\ldots,v_n)$ of 
$\sigma$ may then be deleted or remain undefined.

\begin{figure}[ht]
    \centering
    \includegraphics[scale=0.63]{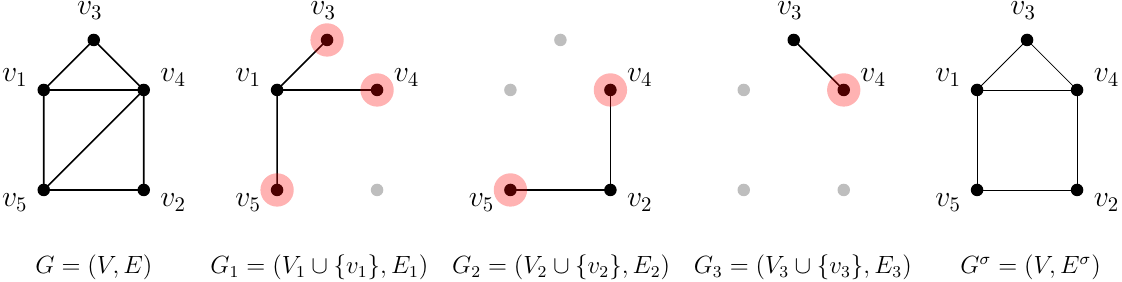}
    \caption{A graph $G$ with  an ordering $\sigma : = (v_1, \ldots, v_5)$ of $V(G)$, the corresponding $G_i$ for $i \in [3]$, and  $G^\sigma$. For each $G_i$ ($i \in [3]$), the vertices of $V_i$ are marked with red disks. Note that $v_4$, the only vertex of $V_3$, is not contained in $N_G(v_4)$, so  $E_4 = \emptyset$.}
    \label{fig:example_G_sigma}
\end{figure}
We refer to an interval-order subgraph of a graph $G$ as \emph{maximal} if it is inclusion-wise maximal with respect to its edge-set.
\begin{corollary}\label{cor:co_intervals} Let $G=(V,E)$ be a graph. Then, for any ordering $\sigma$ of $V$, $G^\sigma$ is an interval-order subgraph of $G$.  Conversely, any maximal interval-order subgraph of $G$ is  $G^\sigma$ for some ordering $\sigma$ of $V$. 
\end{corollary}
\begin{proof}   
    First, we show that $G^\sigma = (V, E^\sigma)$ is an interval order.
    By construction, the ordering $\sigma$ of $V$ ensures that (\ref{eq:chain_property}) is satisfied.
    Thus, the first part of the corollary follows directly from Lemma~\ref{lem:co_intervals}.
    For the converse, let $H = (V, F)$ be a maximal interval-order graph if $G$.
    By the reverse implication of Lemma~\ref{lem:co_intervals}, there exists an ordering $\sigma$ satisfying (\ref{eq:chain_property}) (for $H$).
    As $F \subseteq E^\sigma$, by the maximality of $H$, we have $H = G^\sigma$.
    \qed 
\end{proof}

The following lemma helps to reduce the number of orderings of the vertex set to consider for generating maximal interval-order subgraphs.
\begin{lemma}\label{lem:det}  
Let $G=(V,E)$ be a graph, $G^\sigma$ a maximal interval-order subgraph of $G$, 
and $\sigma_i = (v_1, \ldots, v_i)$ an $i$-prefix of $\sigma$. Then:
\begin{itemize}
\item[(i)] For $u, v\in V \setminus \{v_1, \ldots, v_i\}$, if $N(u)\supseteq N(v)\cap N^+(v_i)$, then $G^\sigma=G^{\sigma'}$, where $\sigma'$ is derived from $\sigma$ by possibly swapping $u$ and $v$, so that $u$ precedes $v$.
\item[(ii)] In each  ordering $\sigma'$ with  prefix $\sigma_i$ such that $G^\sigma=G^{\sigma'}$, $\sigma_i$ is followed by all $N_i \coloneqq \{v \in V \setminus \{v_1, \ldots, v_i\} :  N(v)\supseteq N^+(v_i)\}$ in arbitrary order.
\item[(iii)]
If $N^+(v_i)= \{u,v\}$, then $\sigma_i$ is immediately followed by $(N(u) \cap N(v)) \setminus \{v_1, \ldots, v_i\}$ and then by the remaining vertices   of $N(u)$ or of $N(v)$.
\end{itemize}
\end{lemma}

Note that $N_i$ is exactly the set of elements electable to be added to $\sigma_i$ as $v_{i+1}$ so that $N^+(v_{i+1})= N^+(v_{i})$.
\begin{proof}[of Lemma \ref{lem:det}]
    To check $(i)$, note that swapping $u$ and $v$, the condition of $(i)$ makes sure that neither $N^+(u)$ nor $N^+(v)$ decrease. 

   We check now $(ii)$. If $\sigma_i$ is not immediately followed by the vertices of $N_i$, modify $\sigma$ by moving all the vertices of $N_i$ immediately after $v_i$, in arbitrary order.
    Clearly, each $N^+(v_i)$ is replaced by a superset, and at least one of them by a proper one: $E^\sigma$ increases, contradicting maximality.
	
	To prove $(iii)$, denote by  $N_u$, $N_v$, and $N_{u,v}$ the vertices in $V\setminus \{v_1, \ldots, v_i\}$ that are adjacent only to $u$, only to $v$,  or to both, respectively.
    Observe that $N_{u,v} = N_i$, so, by $(ii)$, $\sigma_i$ is immediately followed by $N_{u,v}$.
    Then, by $(i)$, either an element of $N_u$ or an element of $N_v$ follows, unless both are empty. 
	Finally, applying $(ii)$ again, the remaining vertices of $N_u$ or of $N_v$ must follow. 
\qed 
\end{proof}

\section{Interval orders in line graphs}\label{sec:interval_orders}
In this section, we characterize the maximal interval-order subgraphs of $L(K_n)$.
This characterization is crucial for computing the boxicity of $\KG(n, 2)$ (Theorem~\ref{thm:box_kneser}), addressing the IGC for complements of line graphs (Lemma~\ref{lem:all_minimal_completion} and Theorem~\ref{thm:minimum_completion}), and determining the boxicity of complements of line graphs (Theorem~\ref{thm:box_co_line}).
Solving these problems requires refined insight into the interval completions of the considered graph (or equivalently, into the interval orders of its complement) already for graphs of small order, such as the Petersen graph.

To describe the edge-sets of the maximal interval-order subgraph of $L(K_n)$, we denote by $(e,f)$  the edge of $L(K_n)$ between two incident edges $e,f\in E(K_n)$ and, for simplicity, borrow notation from $K_n$: 
\begin{itemize}
    \item[-]
    for a vertex $v\in V(K_n)$, $Q_v$  denotes the set of $\frac{(n-1)(n-2)}2$ edges of the clique formed in $L(K_n)$ by the $n-1$ edges of $\delta(v) \subseteq E(K_n)$ as vertices of $L(K_n)$;
    \item[-]
    for an edge $e=uv\in E(K_n)$, $\delta_{uv}$ denotes the set of edges $ef\in E(L(K_n))$, where $f\in V(L(K_n))$ is incident to $u$ or $v$ in $K_n$, that is, $\delta_{uv}$ is the edge-set of the star of centre $e=uv$ in $L(K_n)$, $|\delta_{uv}|= 2(n-2)$;
    the set $\delta_{uv^-}$ consists only of the edges $ef$, where $f$ is incident to $u$ in $K_n$, $|\delta_{uv^-}|= n-2$;
    \item[-]
    for a set $U\subseteq V(K_n)$, $K_U$ denotes the edge-set of $L(K_n[U])$; $K_{u,v,w^-}\coloneqq\{(uv,uw), (uv,vw)\}$.  
\end{itemize}
Now, we can state our characterization.

\begin{lemma}\label{lem:maximal_co_interval} 
For any $n\geq5$ and any maximal interval-order subgraph $H$ of $L(K_n)$, there exist distinct vertices $a, b, c, d, e \in V(K_n)$
such that $E(H)$ corresponds to one of the following subsets of $E(L(K_n))$:
\begin{align}
    E_{a,b,c,d,e} &= Q_a\cup \delta_{ab}\cup \delta_{ad}\cup K_{\{a,b,c\}}\cup K_{\{a,d,e\}},  \label{eq:typeA} \tag{a} \\
    E_{a,b,c,d} &=\delta_{ab}\cup \delta_{ad}\cup K_{\{a,b,c,d\}}, \label{eq:typeB} \tag{b} \\
    F_{a,b,c,d} &= \delta_{ab} \cup \delta_{ad}\cup  \delta_{ac^-}\cup K_{\{a,b,c\}}\cup K_{\{a,b,d\}}\cup K_{a,d,c^-}\cup K_{b,c,d^-} \label{eq:typeC} \tag{c}\\
        & \textrm{or $F'_{a,b,c,d}$, where } \delta_{ad} \textrm{ is replaced by } \delta_{bc}. \notag
\end{align}
In particular,
\[|E_{a,b,c,d,e}|= \frac{(n+2)(n-1)}2,\, |E_{a,b,c,d}|= 4(n-1),\, |F_{a,b,c,d}|= |F'_{a,b,c,d}|= 5(n-2).\]
\end{lemma}

\begin{proof}
Let $H$ be a maximal interval-order subgraph of $L(K_n)$.
By Lemma~\ref{lem:co_intervals}, there is an ordering $\sigma = (v_1, \ldots, v_{m})$ of $V(L(K_n)) = E(K_n)$ so that $E(H) = E^\sigma$.
We prove that $E^\sigma$ has one of the forms (\ref{eq:typeA}), (\ref{eq:typeB}), or (\ref{eq:typeC}).

Let  $v_1\coloneqq ab\in E(K_n)$ be the first vertex in $\sigma$,  and $i(\sigma)$  be the first index of $\sigma$ such that $v_i$, as an edge of $K_n$, is not incident to $a$.
Clearly, $i(\sigma) \leq n$.
%

\medskip

\noindent {\bf Case 1}: $i(\sigma)\ge 4$, that is, $v_1, v_2, v_3\in E(K_n)$ are incident to  $a$ (Fig.~\ref{fig:max_int_Case1}).

Denote the two endpoints of $v_{i(\sigma)}$ by $d$ and $e$.

\begin{claim}
    $i(\sigma)=n-2$ and, in particular, $N^+(v_{i(\sigma)})=\{ad,ae\}$.
\end{claim}
Indeed, if $i(\sigma)\ge n-2$, then $N^+(v_{n-3}) = \{ad', ae'\}$ contains only the two edges in $E(K_n)$ that are incident to $a$ but not yet selected in $\sigma$.
    By Lemma~\ref{lem:det}~$(ii)$, $\sigma_{n-3}$ must continue with $v_{n-2} = d'e'$ implying that $i(\sigma)\leq n-2$, $\{d, e\} = \{d', e'\}$, and $N^+(v_{n-2})=\{ad,ae\}$.

    For each edge $af$ (if any), different from $v_1,\ldots, v_{i(\sigma)-1}$ and from $ad, ae$, the neighbourhood $N(af)$ contains $N(de) \cap N^+(v_{i(\sigma)-1})$.
    Hence, applying Lemma~\ref{lem:det}~$(i)$, each edge of this type precedes $de$ and $i(\sigma) = n-2$.
    This concludes the proof of the claim.

	\begin{figure}[ht]
		\centering
		\includegraphics[scale=0.75]{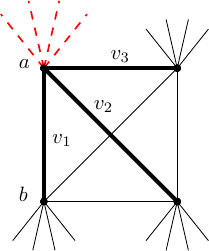}
		\caption{The edges of $V(K_n)$ that can be contained in $\sigma$ in Case 1. 
        %
        Dashed red lines represent the vertices in $N^+ (v_3)$.
        }
		\label{fig:max_int_Case1}
	\end{figure}
  By Lemma~\ref{lem:det}~$(iii)$, the condition $N^+(v_{i(\sigma)})=\{ad,ae\}$ implies that there are only two possible ways to continue: either with the remaining edges of $K_n$ incident to $ad$,  or with that incident to $ae$, in arbitrary order. 
    
   We now give a closer look at the edges in $E^\sigma$.
    Let $c$ be the second vertex of $K_n$ in $v_2$, i.e. $v_2 = ac$, and assume that the $i(\sigma)$-prefix of $\sigma$ is followed by the edges incident to $ad$.

    First, starting with $v_1=ab$ and $v_2=ac$  adds already $\delta_{ab}\cup \delta_{ac^-} \cup K_{\{a,b,c\}}$ to $E^\sigma$, consisting of $3(n-2)$ edges: $2(n-2)$ edges in $\delta_{ab} \subseteq Q_a \cup Q_b$, $n-3$ edges in $\delta_{ac-} \subseteq Q_a$ (the edge between $ab$ and $ac$ has already been counted) and in addition, the edge between $ac$   and   $bc$, which is neither in $Q_a$, nor in $Q_b$.
    Then $v_3,\ldots, v_{n-3}$ is the list of all other edges of $K_n$ incident to  $a$,  except $ad, ae$,  and $v_{n-2}=de$.
    We continue by adding the remaining $n-3$ edges incident to $d$ different from $ad$, in arbitrary order.
    Finally, $v_{2(n-2)}\coloneqq ae$.
    Overall, we added to $E^\sigma$ the edges of $Q_a$, and besides that, the $n-2+1$ edges of $\delta_{ad}\cup K_{\{a,d,e\}}$ (respectively, $\delta_{ae}\cup K_{\{a,d,e\}}$), symmetrically to the effect of the starting two edges.
    We have that $E^\sigma$ corresponds to $E_{a,b,c,d,e}$ in~(\ref{eq:typeA}); moreover
    $$ |E^\sigma| =\frac{(n-1)(n-2)}2 + 2(n-1) = \frac{(n+2)(n-1)}2. $$

	\medskip
 
	\noindent{\bf Case~2}: $i(\sigma)=3$.
	
	Then $v_2=ac\in E(K_n)$ for some $c\in V(K_n)$ different of $a$ and $b$, 
	and the assumption $i=3$ restricts our choices to the following three possibilities for $v_3$: 
	$v_1$, $v_2$, $v_3$ form the three edges of a triangle of $K_n$, then $v_3=bc$; or $v_1$, $v_2$, $v_3$ form a path, that is, $v_3=bd$, or  $v_3=cd$, so that $a,b,c,d$ are four different vertices of $K_n$; or else, $v_3=ed$, where $a,b,c,d,e$ are five different vertices of $K_n$ (Fig.~\ref{fig:max_int_Case2}).
	\begin{figure}[ht]
		\centering
		\includegraphics[scale=0.75]{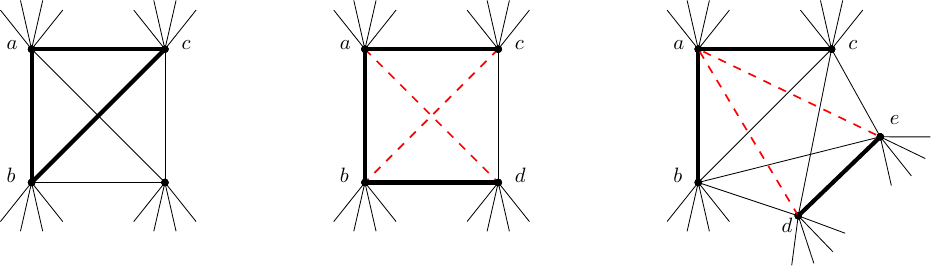}
		\caption{The three edges of $V(K_n)$ that $\sigma$ may start within Case 2. }
		\label{fig:max_int_Case2}
	\end{figure}
	
	\smallskip {\bf Case 2.1}: $v_1$, $v_2$, $v_3$ form the three edges of a triangle of $K_n$ (Fig. \ref{fig:max_int_Case2} (left)).
	
	In this case, $N^+(v_3) = \emptyset$.
    Thus, any completion of the ordering starting with $v_1, v_2$ will be contained in $\delta_{ab}\cup\delta_{ac}$, a subset of (\ref{eq:typeA}), (\ref{eq:typeB}), or (\ref{eq:typeC}), contradicting the maximality of $H$.
	
	\smallskip {\bf Case 2.2}:  $v_1$, $v_2$, $v_3$ form a path (Fig. \ref{fig:max_int_Case2} (centre)).
	
	Then $v_3=bd$ (or, equivalently $v_3 = cd$), where $d\in V(K_n)$ is different from $a,b,c$.
	The set of common neighbours of $v_1$, $v_2$, $v_3$ is $\{bc, ad\}$.
    By Lemma~\ref{lem:det}~(ii), $v_4=cd$ (or, equivalently $v_4 = bd$, if $v_3 = cd$), and we can either continue the construction of an ordering $\sigma$ with all neighbours in $L(K_n)$ of $ad$,  or all neighbours of $bc$ Lemma~\ref{lem:det} (iii), arriving respectively at: 
	\[\delta_{ab}\cup \delta_{ad}\cup  \delta_{ac^-}\cup K_{\{a,b,c\}}\cup K_{\{a,b,d\}}\cup K_{a,d,c^-}\cup K_{b,c,d^-},\] 
	\[\delta_{ab}\cup \delta_{bc}\cup  \delta_{ac^-}\cup K_{\{a,b,c\}}\cup K_{\{a,b,d\}}\cup K_{a,d,c^-}\cup K_{b,c,d^-}.\] 
    In the first case, we get $E^\sigma=F_{a,b,c,d}$ and in the second case, $E^\sigma=F'_{a,b,c,d}$. Therefore,  Case~2.2 leads anyway to an interval-order graph of the form of (\ref{eq:typeC}).

    We can now compute $|F_{\{a,b,c,d\}}|$, $|F'_{\{a,b,c,d\}}|$  for instance by counting the size of the out-neighbourhood of each $v_i$:
    $|N^+(v_1)| + |N^+(v_2)|= 3(n-2)$, as in Case 1;
    $|N^+(v_3)|=|N^+(v_4)|=2$;  then we add $2(n-4)$ edges of $L(K_n)$ with out-neighbourhoods of size $1$, after which the out-neighbourhoods are empty, so   $|F_{\{a,b,c,d\}}|= |F'_{\{a,b,c,d\}}|= 3(n-2)+4+2(n-4)=5(n-2)$.
 
	\smallskip {\bf Case 2.3}: $v_3=de$, where $a,b,c,d,e \in V(K_n)$ are five different vertices (Fig.~\ref{fig:max_int_Case2}~(right)).

	In this case $N^+(v_3)=\{ad, ae\}$, so by Lemma~\ref{lem:det}~$(i)$, all the edges incident to $a$, except from $ad$ and $ae$, precede $v_{i(\sigma)}$, that is,  $n-2=i(\sigma)=3$. Hence $n=5$ and, continuing the ordering $\sigma$ as in Case 1, we get the edge-set (\ref{eq:typeA}).

    \medskip
 
	\noindent {\bf Case~3}: $i(\sigma)=2$, that is, $\{v_1,v_2\}$ is a matching (Fig. \ref{fig:max_int_Case3} (left)).

    Note that if the edges corresponding to $v_1$ and $v_2$ are incident, then a common point of $v_1$ and $v_2$ can be chosen as $a$ yielding $i(\sigma) > 2$.
    Hence, we can assume that $v_2=cd$, and $a,b,c,d\in V(K_n)$ are different.
    So $N^+ (v_2)= N(v_1)\cap N(v_2) = \{ad,db,bc, ca\}$, and any choice of $v_3$ corresponding to an edge in $K_n$ incident to $a,b,c$, or $d$ leads to $|N^+(v_3)|=2$ (Fig.~\ref{fig:max_int_Case3}).
	There are two ways of continuing then (similarly to Cases 2.2 and 2.3):
	\begin{figure}[ht]
		\centering
		\includegraphics[scale=0.75]{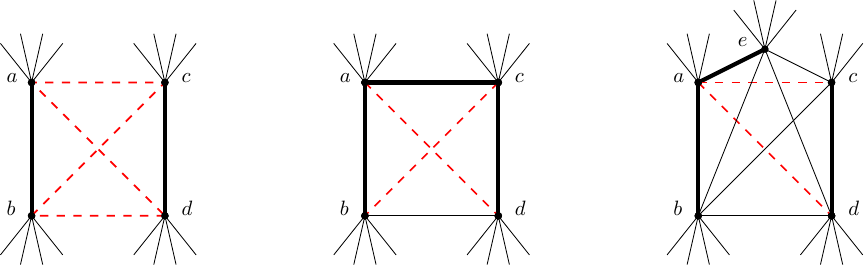}
		\caption{An illustration of Case 3 and its two subcases.}
		\label{fig:max_int_Case3}
	\end{figure}

    \smallskip {\bf Case 3.1}: $v_1,v_2,v_3$ form a path (Fig. \ref{fig:max_int_Case3} (centre)).

    Then $v_3 = ac$ (or, equivalently, $v_3 = bd)$.
    The set of common neighbours of $v_1$, $v_2$, $v_3$ is $\{bc, ad\}$.
	Thus, by Lemma~\ref{lem:det} (ii), $v_4=bd$ (or, equivalently, $v_4 = ac$), and the size of the out-neighbourhood remains $2$.
    As all edges incident to both $bc$ and $ad$ already appear in the ordering, $N^+(v_5)$ has a size at most $1$.
    By Lemma~\ref{lem:det}~$(iii)$, we can conclude the ordering with all neighbours in $L(K_n)$ of $ad$,  or all neighbours of $bc$, arriving respectively at:
    \[ \delta_{ab}\cup \delta_{ad}\cup K_{\{a,b,c,d\}}, \text{ and } \delta_{ab}\cup \delta_{bc}\cup K_{\{a,b,c,d\}}.\] 
    Hence, $E^\sigma$ is an edge-set of the form (\ref{eq:typeB}).

    Besides the $12$ edges of $K_{a,b,c,d}$, each of $\delta_{ab}$ and $\delta_{ad}$ contain $2(n-4)$  edges of $L(K_n)$, so  $E^\sigma$ has  $12 + 2(n-4) + 2(n-4)  = 4(n-1)$ edges.

    \smallskip{\bf Case 3.2}: $v_3=ae$,  where $a,b,c,d,e\in V(K_n)$ are five different vertices (Fig.~\ref{fig:max_int_Case3} (right)). 
    
	Then $N^+(v_3)=\{ac,ad\}$, which adds the two edges $(ae,ac), (ae, ad)$ to $E^\sigma$.  By Lemma~\ref{lem:det} (ii), $\sigma$ has to continue with the rest of the edges incident to $a$, in arbitrary order, finishing with $v_{n-1} = ac$, $N^+(v_{n-1})=\{ad\}$ (or, equivalently, $v_{n-1} = ad$, if so, swap the label of $c$ and $d$). 
    Applying Lemma~\ref{lem:det} (ii) again, we have to complete $\sigma$ adding the remaining $n-3$ edges of $K_n$ that are incident to $d$, in arbitrary order. 
    We now get the set:

 \[ \delta_{ab}\cup \delta_{ad}\cup  \delta_{ac^-}\cup K_{\{a,c,d\}}\cup K_{\{a,b,d\}} \cup K_{a,b,c^-}\cup K_{c,d,b^-},\] 
 which is, $F_{a,d,c,b}$, 
concluding again with $E^\sigma$ of the form (\ref{eq:typeC}).

\medskip

\noindent
Cases $1$, $2$, and $3$ conclude the proof of the lemma.
\qed
\end{proof}

\section{Boxicity of the Petersen graph and complements of line graphs}\label{sec:petersen}

In this section, we exploit the description of the maximal interval-order subgraphs of $L(K_n)$ (Section \ref{sec:interval_orders}) to prove our main results. We  say that a subgraph  is \emph{ of type (a), (b), or (c) } if its edge-set corresponds to that 
of (\ref{eq:typeA}), (\ref{eq:typeB}), or (\ref{eq:typeC}), respectively; $F_{a,b,c,d}$ and $F'_{a,b,c,d}$ are both considered as of type (c). 

In Section \ref{subsec:proof_main_thm} we establish  Theorem \ref{thm:box_kneser},  that is $\boxi(\overline{L(K_n)}) = n - 2$ for every $n \geq 5$. First, we show the easier upper bound (Lemma \ref{lem:kneser_upper_bound}). The lower bound, that is, the tightness of the proven upper bound, is then proved for $n=5,6$  separately (Lemmas \ref{lem:box_petersen} and \ref{lem:boxlk6}).
These two proofs introduce already the general ideas, but with the difference that for $n\le 6$ the interval-order subgraphs of type  (b) and (c) are larger or equal to those of type (a), so they do play a more essential role.  Then, in the main proof, we address the challenge of extending the proof for  $n\ge 7$.  
Section \ref{subsec:box_co_line} makes one more step for generalizing the results to arbitrary complements of line graphs (without the assumption $G = K_n$), finishing the proofs of Theorems~\ref{thm:minimum_completion} and \ref{thm:box_co_line}. 

\subsection{ Proving Theorem \ref{thm:box_kneser} }\label{subsec:proof_main_thm}
Theorem 1.1 in~\cite{2023_Caoduro} states that the boxicity of the Kneser-graph $\KG(n,k)$ is at most $n-2$ for $n\geq 2k+1$.
We present a simpler proof for $k=2$.

\begin{lemma}\label{lem:kneser_upper_bound}
	For $n\geq 5$, $\boxi(\overline{L(K_n)}) \leq n-2$.
\end{lemma}
\begin{proof}
    Denote the vertices of $V(K_n)$ by $v_1, \ldots, v_n$.
    By Lemma~\ref{lem:maximal_co_interval}, we can construct a family $\Cscr = \{E_1, \ldots, E_{n-2}\}$, where, for each $i\in [n-2]$, $(V(K_n), E_i)$ is an interval-order subgraph of $L(K_n)$ of type (a) and $Q_{v_i} \cup \delta_{v_iv_{n-1}} \cup \delta_{v_iv_{n}} \subseteq E_i$.
    We show that $\Cscr$ covers $L(K_n)$.

    As $E(L(K_n)) = \bigcup_{i=1}^n Q_{v_i}$ and, for $i\in [n-2]$, $Q_{v_i} \subseteq E_i$, we just need to argue that $\Cscr$ also covers $Q_{v_{n-1}}$ and $Q_{v_n}$.
    To do this, note that $Q_{v_{j}}$ ($j\in [n]$) can be covered by $\{\delta_{v_iv_{n-1}} \ : \ i \in [n], \  i \neq j\}$; however, any $n-2$ of these $\delta$  suffice. 
    In particular, $Q_{v_{n-1}}$ and $Q_{v_{n}}$ are covered by the union of $\delta_{v_iv_{n-1}} \cup \delta_{v_iv_{n}} \subseteq E_i$ for $i\in [n-2]$.
\qed
\end{proof}

Now, we focus on the lower bound.
Although the number of edges grows quadratically in the graphs of type (a) and only linearly in those of type (b) and (c), the three sizes are similar for small values of $n$.
Therefore, the arguments for proving the lower bound for $n =5$, $n=6$, and $n \geq 7$ are slightly different.
We handle them separately in Lemmas~\ref{lem:box_petersen}, \ref{lem:boxlk6}, and \ref{lem:boxlk7+}.

Before moving to the proofs, we make the following useful observation.
Let $\{a,b,c\}$ and $\{a',b',c'\}$ be two sets of three distinct vertices of $V(K_n)$. We have
\begin{equation}\label{fact:intersecting_deltas}
\text{$\big ( \delta_{ab} \cup \delta_{ac} \big ) \cap \big ( \delta_{a'b'} \cup \delta_{a'c'} \big) = \emptyset$ if and only if $\{a,b,c\}\cap\{a',b',c'\}=\emptyset$.}
\end{equation}

\begin{lemma}\label{lem:box_petersen}
	$\boxi(\overline{L(K_5)}) \geq 3$.
\end{lemma}
\begin{proof}
	Assume for a contradiction that $\boxi(\overline{L(K_5)}) \leq 2$. Let $\{E_1, E_2\}$ be  a $2$-interval-order cover of $L(K_5)$, and assume that $(V,E_1)$ and $(V,E_2)$ are maximal interval-order subgraphs of $L(K_5)$.
	By Lemma \ref{lem:maximal_co_interval},  $|E_i|=14$  (a), or $|E_i|=15$ (c), or $|E_i|=16$ (b), $(i=1, 2)$. In each of these cases, there are three distinct vertices $a_i,b_i,c_i \in V(K_5)$ such that $\delta_{a_ib_i} \cup \delta_{a_ic_i} \subseteq E_i$ ($i \in \{1,2\}$), (see (\ref{eq:typeA}), (\ref{eq:typeB}), and (\ref{eq:typeC})).
    The set $V(K_5)$ has only $5$ vertices, so $\{a_1,b_1,c_1\}$ and $\{a_2,b_2,c_2\}$ have to intersect, and then $|E_1 \cap E_2| \geq 1$ by~(\ref{fact:intersecting_deltas}).
	
	As $|E(L(K_5))| = 5\binom{4}{2} = 30$, $\{E_1, E_2\}$ forms an interval-order cover only if $E_1$ and $E_2$ have both at least $15$ edges, and at least one of them has $16$ edges. Assume that $|E_1|=16$ ($E_1$ is of type (b)), and $E_2$ has either $16$ or $15$ edges (it is of type (b) or (c)). 
	In both cases, there are two sets of four distinct vertices in $V(K_5)$ defining the edge-set of $E_1$ and $E_2$.
	These two sets have at least three common vertices, say $\{a,b,c\}$.
    It follows that $K_{a,b,c}\subseteq E_1$, and also $K_{a,b,c}\subseteq E_2$ if it is of type (b), or $|K_{a,b,c}\cap E_2|\ge 2$  if it is of type (c).
    Either way, $|E_1| + |E_2| - | E_1 \cap E_2|\le 16+15-2 < 30$, contradicting the assumption that  $E_1\cup E_2=E(L(K_5))$.
    \qed
\end{proof}

\begin{lemma}\label{lem:boxlk6}
	$\boxi(\overline{L(K_6)}) \geq 4$.
\end{lemma}
\begin{proof}
	Assume for a contradiction that $\boxi(\overline{L(K_6)}) \leq 3$. Let $\{E_1,E_2,E_3\}$ be a $3$-interval-order cover of $L(K_6)$, and assume that $(V,E_1)$, $(V,E_2)$, and $(V,E_3)$ are maximal interval-order subgraphs of $L(K_5)$.
    By Lemma \ref{lem:maximal_co_interval}, $|E_i| = 20$ for each $i \in [3]$ and, additionally, there are three distinct vertices $a_i,b_i,c_i \in V(K_6)$ such that $\delta_{a_ib_i} \cup \delta_{a_ic_i} \subseteq E_i$. Both of these observations hold regardless of the type of $E_i$.
	As $|E(L(K_6))| = 6\binom{5}{2} = 60$, $E_1,E_2$, and $E_3$ must be pairwise disjoint. Applying~(\ref{fact:intersecting_deltas}) three times, we deduce that the nine vertices of $a_i,b_i,c_i$ for  $i \in [3]$  are all distinct, contradicting $|V(K_6)| = 6$. 
    \qed
\end{proof}

\begin{lemma}\label{lem:boxlk7+}
	For $n\geq 7$, $\boxi(\overline{L(K_n)}) \geq n-2$.
\end{lemma}
\begin{proof}
    Assume for a contradiction that $\boxi(\overline{L(K_n)}) \leq n-3$. Let $\Escr := \{E_i : i \in [n-3] \}$ be an $(n-3)$-interval-order cover of $L(K_n)$, and assume that $\{(V,E_i) : i \in [n-3]\}$ are maximal interval-order subgraphs of $L(K_n)$.
    By Lemma \ref{lem:maximal_co_interval}, for each $i \in [n-3]$, there are distinct $a_i, b_i,c_i \in V(K_n)$ such that: 
    \begin{itemize}
        \item [-] $E_i \supset Q_{a_i} \cup \delta_{a_ib_i} \cup \delta_{a_ic_i}$, if $E_i$ is of type (a), or
        \item [-] $E_{i} \supset \delta_{a_{i}b_{i}} \cup \delta_{a_{i}c_{i}}$, if $E_i$ is of type (b) or (c).
    \end{itemize}
    
	First, observe that \textit{at least $n-4$ interval-order graphs in this cover are of type (a)}.
	Indeed, if there are at most $n-5$ edge-sets of type (a), then, by Lemma~\ref{lem:maximal_co_interval},
    the number of covered edges is at most  $$(n-5) \frac{(n+2)(n-1)}{2} + 10(n-2),$$
	quantity that, for $n\geq 7$, is strictly smaller than $|E(L(K_n))| = n\binom{n-1}{2}$.
	Therefore, the only two cases to consider are: the cover contains $n-3$ edge-sets of type (a), or it contains $n-4$ edge-sets of type (a) and one edge-set of type (b) or (c).
    Note that $\sum_{i=1}^{n-3} |E_i|$ equals to $ |E(L(K_n))|+ \frac{(n-1)(n-6)}{2}$ in the first case, and $|E(L(K_n))|+ n-6$ in the second case.
    In both cases, we show that for any possible assignment of $\{(a_i,b_i,c_i) : i \in [n-3]\}$, the number of edges in $\bigcup_{i \in [n-3]} E_i$ is strictly smaller than $|E(L(K_n))|$.
    
	\medskip 	\noindent
    {\bf Case 1:} $E_i$ is of type (a)  for all $i\in[n-3]$.
	
    For any vertex $v\in V(K)$, the edge-set $Q_{v}$ has $\frac{(n-1)(n-2)}{2} > \frac{(n-1)(n-6)}{2}$ edges.
    As $Q_{a_i} \subseteq E_i$ ($i\in [n-3])$ and the intersection between two sets in $\Escr$ cannot be larger than $\frac{(n-1)(n-6)}{2}$, the elements in  $\{a_i : i \in [n-3]\}$ must be pairwise distinct.
    Thus, we can assume without loss of generality that for all $i \in [n-3]$, $a_i = v_i$.

    Then, if $b_i$ ($i \in [n-3]$) coincides with $a_j = v_j$ for some $j\in [n-3]$, the edges in $\delta_{a_ib_i}$ are already part of the covering because $\delta_{a_ib_i} = \delta_{a_ia_j} \subseteq Q_{a_i} \cup Q_{a_j}$.
    Hence, by re-assigning $b_i$ to a vertex in $\{v_{n-2},v_{n-1},v_{n}\}$, we can only enhance the covering. 
    As the same argument also applies if $c_i = v_j$ for $j\in [n-3]$, we can further assume without loss of generality that for all $i \in [n-3]$, $b_i, c_i \in \{v_{n-2},v_{n-1},v_{n}\}$.

    By the previous assumption, it follows that for each pair of distinct indices $i,j \in [n-3]$ the sets $\{b_i,c_i\}$ and $\{b_j,c_j\}$ intersect; 
    let $x_{ij}$ be a vertex in the intersection.
    For a fixed $i \in [n-3]$, the set
    $$\{ (v_i x_{ij}, v_j x_{ij}) \in E_i \cap E_j : j \in [i-1] \}$$
    contains $i-1$ distinct edges, so
    $E_i$ can add at most $|E_i| - (i-1)$ new edges to $\bigcup_{ j \in [i-1]} E_j$.
    We finally reach a contradiction:
    $$\Big | \bigcup_{i \in [n-3]} E_i \Big | \leq |E(L(K_n))| + \frac{(n-1)(n-6)}{2} - \sum_{i=2}^{n-3} (i-1) < |E(L(K_n))|.$$

	\medskip
	\noindent {\bf Case 2:} Each interval-cover graph of the cover but one, is of type (a).

    Assume that $E_i$ is of type (a) for $i \in [n-4]$, and $E_{n-3}$ is of type (b) or (c). 
    Considering the vertices $\{(a_i,b_i,c_i) : i \in [n-3]\}$ partially defining the edge-sets of $\Escr$, we can make the following assumptions:
    \begin{itemize}
        \item[$(i)$] for all $i \in [n-4]$, we have $a_i = v_i$,
        \item[$(ii)$] for all $i \in [n-4]$, we have $\{b_i, c_i\} \subseteq \{v_{n-3}, v_{n-2}, v_{n-1}, v_{n}\}$, and
        \item[$(iii)$] for $i = n-3$, we have $a_{n-3} = v_{n-3}$, $b_{n-3} = v_{n-2}$, and $c_{n-3} = v_{n-1}$.
    \end{itemize}
	The arguments for $(i)$ and $(ii)$ are analogous to those presented in Case~1.
    To see $(iii)$, observe that $a_{n-3}, b_{n-3}, c_{n-3}$ have to be distinct from the previous $a_i = v_i$ because otherwise there would exist $i$ such that $$ |E_i \cap E_{n-3}| \geq |Q_{a_i} \cap \left ( \delta_{a_{n-3}b_{n-3}}\cup \delta_{a_{n-3}c_{n-3}} \right )| \geq n-2,$$
    and then $\Escr$ would not have enough edges to cover $L(K_n)$.

    By assumptions $(ii)$ and $(iii)$, for all $i \in [n-4]$, we have $\{a_i,b_i,c_i\} \cap \{a_{n-3}, b_{n-3}, c_{n-3}\} \neq \emptyset$.
    Hence, (\ref{fact:intersecting_deltas}) yields $|E_i \cap E_{n-3}| \geq 1$.
	This means that $E_i$ can add at most $|E_i| - 1$ new edges to $E_{n-3}$.
	We now reach a contradiction:
	\[\Big | \bigcup_{i \in [n-3]} E_i \Big | \leq |E(L(K_n))| + (n-6) - \sum_{i=1}^{n-4} 1 < |E(L(K_n))|. \tag*{\qed}\]
\end{proof}

\subsection{Interval completion and boxicity of complements of line graphs}\label{subsec:box_co_line}

Lemma \ref{lem:maximal_co_interval} presents all the maximal interval-order subgraphs of $L(K_n)$.
Now, we show how to use this information to generate all the maximal interval-order subgraphs of $L(G)$ for any graph $G=(V, E)$. 
This will prove Lemma~\ref{lem:all_minimal_completion} because complementing the maximal interval-order subgraphs of $L(G)$, we get the minimal interval completions of $\overline{L(G)}$.
We then derive Theorems~\ref{thm:minimum_completion} and~\ref{thm:box_co_line} as easy applications of Lemma~\ref{lem:all_minimal_completion}.

\begin{proof}[of Lemma \ref{lem:all_minimal_completion}]
    Consider the complete graph $K_V = (V, \binom{V}{2})$.
    Let $\Ascr$  be the family of maximal interval-order subgraphs of $L(K_V)$,
    and define $\Bscr \coloneqq \{G_A[E]: G_A\in\Ascr\}$.
    As interval-order graphs are closed under taking induced subgraphs, also the graphs in $\Bscr$ are interval orders.

    We now show that all maximal interval-order subgraphs of $L(G)$ are contained in $\Bscr$.
    Let $H = (E, B)$ be a maximal interval-order subgraph of $L(G)$.
    As $H$ is also a subgraph of $L(K_V)$, there exists a maximal interval order $H' \in \Ascr$  of $L(K_V)$ such that the edge-set of $H'$ contains $B$.
    Consider the induced subgraph $H'[E]$ of $H'$, and let $B'$ be its edge-set.
    Clearly, $B \subseteq B'$.
    Moreover, as $H'[E]$ is an interval-order subgraph of $L(G)$, the maximality of $H$ yields $B' \subseteq B$.
    Therefore, $H = H'[E]$, and so $H \in \Bscr$.
   
    By Lemma~\ref{lem:maximal_co_interval}, $|\Ascr| \leq n^5$.
    Moreover, each set in $\Ascr$ can be computed in $\Oscr(n^2)$ time given the descriptions of the maximal edge-sets in Lemma~\ref{lem:maximal_co_interval}.
    Computing the induced subgraph with respect to the vertex set $E$ requires at most an additional $\Oscr(n^2)$ time per graph. 
    Hence, the total time to compute $\Bscr$ is $\Oscr(n^7)$.
    \qed
\end{proof}

Given Lemma \ref{lem:all_minimal_completion} and the notion of $k$-interval-order covers, the proofs of Theorems~\ref{thm:minimum_completion} and \ref{thm:box_co_line}  follow easily.

\begin{proof}[of Theorem \ref{thm:minimum_completion}]
    Let $G$ be a graph on $n$ vertices.
    By Lemma~\ref{lem:all_minimal_completion}, all inclusion-wise minimal interval completions of $\overline{L(G)}$ can be listed in $\mathcal{O}(n^7)$ time.
    We select the one with the least number of edges.
\qed
\end{proof}

\begin{proof}[of Theorem \ref{thm:box_co_line}]
If $k\geq n-2$, the solution to the decision problem is trivial by Theorem \ref{thm:box_kneser}, so we can assume $k \leq n-3$.
By Lemma \ref{lem:all_minimal_completion},  the family $\Bscr$ containing all maximal interval-order subgraphs of $L(G)$ has cardinality $\Oscr(n^5)$, and all its members can be presented in   $\Oscr(n^7)$ time. To decide if $\boxi(\overline{L(G)}) \leq k$, one can simply generate all $k$-tuples of elements of $\Bscr$ and check, for each of them, if it covers $L(G)$.
This takes
$\Oscr(n^7) + \Oscr(n^{5k})f(n,k)$ time where $f(n,k)$ is the time of checking whether the sets in a $k$-tuple of $\Bscr$ cover all edges of $L(G)$.
As $L(G)$ has at most $n^3$ edges and $k<n$, $f(n,k) = \Oscr(n^4)$.
\qed
\end{proof}

\section{Boxicity of line graphs}\label{sec:line_graphs}

In this section we continue relying on interval-order subgraphs of a graph for determining the boxicity of their complement.
In Section~\ref{sec:prelim}, we applied this technique to efficiently determine the boxicity of complements of line graphs.
Our focus here shifts to interval-order subgraphs of complements of line graphs for bounding the boxicity of line graphs.
In particular, in Sections~\ref{subsec:upper_bound_line_graph} and~\ref{subsec:lower_bound_line_graph}, we establish bounds on the boxicity of the line graph of the complete graph: an upper bound (Theorem~\ref{thm:line_upper}) and a lower bound (Theorem~\ref{thm:lower_bound_box_line_graph}), respectively.

\medskip

First, note that for exhibiting the maximal  interval-order subgraphs of  the  complement of a line graph, the permutations of the vertex-set of $L(G)$ can be replaced by the much simpler permutation on $V(G)$:

Let $G=(V,E)$ be a graph, $n\coloneqq|V|$, and $m\coloneqq|E|$.  
By Corollary~\ref{cor:co_intervals}, the family of the maximal interval-order subgraphs of $\overline{L(G)}$ are exactly the graphs $\overline{L(G)}^\sigma$ determined by orderings $\sigma: E\rightarrow [m]$  of its vertices, i.e. the edges of $G$.
However,  we show that several of these orders $\sigma$ define the same graph $\overline{L(G)}^\sigma$, and a family of orderings defining the same graph can be characterized by permutations $\pi: V \rightarrow [n]$ of $V$. 

Let $\pi: V\rightarrow [n]$ and denote  $\Sigma(\pi)$  the set of orderings $\sigma: E\rightarrow [m]$ of $E$, where the edge $e=uv\in E$  precedes $f=wz$ if $\max\{\pi(u), \pi(v)\}<\max\{\pi(w),\pi(z)\}$.
In cases of equality, that is, for the set of edges $F_v$ incident to a vertex $v$ such that $\pi(v)$ is larger than the $\pi$-value of their other endpoint, we add $F_v$  to $\Sigma(\pi)$ in all possible orders between the elements of $F_v$.  Indeed, all these orders define the same graph.

For example, consider $n = 5$ and the permutation $\pi = (v_1,v_2,v_3,v_4,v_5)$. 
Then the order 
\[
\big ( v_1v_2, \quad v_1v_3, v_2v_3, \quad v_1v_4, v_2v_4, v_3v_4, \quad v_1v_5, v_2v_5,v_3v_5,v_4v_5 \big )
\]
is in  $\Sigma(\pi)$, and so are all the orders obtained by arbitrarily permuting the edges within each visibly separated group.

The definition of  $\Sigma(\pi)$ corresponds to the more informal description of adding the vertices of $V$ one by one, in increasing order of $\pi$, and when a vertex $v$ is added, all the edges from the already present vertices to the newly added vertex are added to the list in all possible orders.

\begin{lemma}\label{lem:maximal_interval_order_co_line_graph}
	Let $G=(V,E)$ be a graph. Then for any maximal interval-order subgraph $H$ of $\overline{L(G)}$ there exists a permutation $\pi$ of $V$ such that   $H=\overline{L(G)}^\sigma$ for all  $\sigma\in \Sigma(\pi)$.   
\end{lemma}

\begin{proof}
Note first that for any permutation $\pi$ of $V$, all orders $\sigma\in \Sigma(\pi)$ define the same interval-order graph.
This follows from the fact that all edges of $G$ with the same maximum endpoint (relative to $\pi$) have an identical out-neighbourhood in $\overline{L(G)}$, making the order of their appearance irrelevant.

Now, let $n\coloneqq|V|$, $m\coloneqq|E|$, and $H$ be a  maximal interval-order subgraph of $\overline{L(G)}$.
By Corollary~\ref{cor:co_intervals}, there exists an ordering $\sigma$ of $E$  such that $H=\overline{L(G)}^\sigma$. It remains to prove that there exists $\pi: V\rightarrow [n]$ so that $\sigma\in\Sigma(\pi)$.

Define $\pi: V\rightarrow [n]$ with the order of vertices where vertex $u$ precedes vertex $v$ if  $u$ occurs before $v$ as an endpoint of an edge
in the series $\sigma^{-1}(1),  \ldots, \sigma^{-1}(m)$.
If they occur at the same time (e.g. for the two vertices in $\sigma^{-1}(1)$), the order of $u$ and $v$ can be chosen arbitrarily.

For each index $i \in [m]$ for which there is a vertex $v\in V$ occurring for the first time in $\sigma$, consider the set $N_i$  defined in Lemma~\ref{lem:det}~(ii).
Observe that $N_i$ consists of the set of edges of $E$ (vertices of $\overline{L(G)}$) having $v$ as one of their endpoints, where the other endpoint has occurred earlier.
Therefore, by Lemma~\ref{lem:det}~(ii), the $i$-prefix of $\sigma$ is immediately followed by a permutation of $N_i$
 in $\sigma$.
Continuing in this way (or by induction), we get that $\sigma$ is an order in $\Sigma(\pi)$. 
\qed
\end{proof}

\subsection{Proving Theorem \ref{thm:line_upper}}\label{subsec:upper_bound_line_graph}
We first prove that a significant proportion of each subgraph of the complement of a line graph can be covered by just one interval-order subgraph.
\begin{lemma}\label{lem:line_upper_stronger}
    Let $G$ be a graph and $H$ a subgraph of $\overline{L(G)}$ having $k$ edges.
    Then there exists an interval-order subgraph of $\overline{L(G)}$ containing at least $\frac{k}{3}$ edges of $H$.
\end{lemma}
\begin{proof}
    In order to define the desired interval-order subgraph of  $\overline{L(G)}$, sample a permutation $\pi$ uniformly at random from the $|V(G)|!$ permutations of the vertices of $G$ and consider a permutation $\sigma \in \Sigma(\pi)$  of the edges of $G$.
    
	Let $xy$  be an edge of $H$, where $x = ab, y= cd$ are two edges of $G$.
    As $xy$ is also an edge of $\overline{L(G)}$, $a, b, c,$ and $d$ are distinct.
    Then, by the definition of $\overline{L(G)}^\sigma$, the edge $xy$ occurs in $\overline{L(G)}^\sigma$ if and only if  $\max\{\pi(a),\pi(b)\} <  \min\{\pi(c),\pi(d)\}$ or $\max\{\pi(c),\pi(d)\} <  \min\{\pi(a),\pi(b)\}$, that is, in   $8$ of the $24$ permutations of the letters $a, b, c, d$ (whose occurrence in $\pi$  is equally likely).
    Therefore, the probability of the occurrence of $xy$ in $\overline{L(G)}^\sigma$ is exactly $1/3$.
    This yields that the average number of edges of $H$ in a permutation $\sigma' \in \Sigma(\pi')$, where $\pi'$ is a permutation of $V(G)$, is $|E(H)|/3$, and so there exists a permutation $\pi^*$ such that the interval order $\overline{L(G)}^\sigma$,  $\sigma\in\Sigma(\pi^*)$ contains at least one-third of the edges of $H$.
    \qed
\end{proof} 

Note that the claimed interval-order subgraph of $\overline{L(K_n)}$ is not necessarily a subgraph of $H$.
To see this, consider the vertex-set $V \coloneqq \{v_1, \ldots, v_{12}\}$ and the graph $\overline{L(K_V)}$.
The edge-set of $\overline{L(K_V)}$ contains a matching of six edges, $\{v_1v_2, \ldots, v_{11}v_{12}\}$ that can be partitioned into three pairs of edges  $\{v_1v_2, v_3v_4\},$ $\{v_5v_6, v_7v_8\}, $ $\{v_9v_{10}, v_{11}v_{12}\}$ forming in turn a $3$-matching in $\overline{L(K_V)}$.
Let $H$ be this $3$-matching. An interval-order subgraph of $H$ cannot contain two edges of $H$ because an interval-order graph does not contain a $2K_2$, i.e., an induced matching of size $2$.
However, the permutation  $\pi=(v_1, \ldots, v_{12})$ defines an interval-order subgraph $\overline{L(K_V)}^\sigma$ ($\sigma \in \Sigma(\pi)$) of $\overline{L(K_V)}$ which contains all edges of $H$.

\smallskip

We are now ready to prove Theorem~\ref{thm:line_upper}.
\begin{proof}[of Theorem~\ref{thm:line_upper}]
We bound $\boxi(L(G))$ by constructing an interval-order cover of $\overline{L(G)}$ of size at most $\lceil 5 \log n \rceil$.
For this, apply  Lemma~\ref{lem:line_upper_stronger} iteratively $k$ times and delete each time the edges covered by the current interval-order graph.   As $|E(\overline{L(G)}| \leq n^2$, at most $n^2(\frac23)^k$ of the edges remains. The procedure stops latest when  $n^2 ({\frac23})^k <1$, that is, when $k >\frac{2 \log n}{\log 3 - \log 2}$, where $\frac{2}{\log 3 - \log 2} < 5$.
\qed
\end{proof}

\subsection{Proving Theorem~\ref{thm:lower_bound_box_line_graph}}\label{subsec:lower_bound_line_graph}
Let $\pi$ be a permutation of the set $V$.
We denote the restriction of $\pi$ to $T\subseteq V$ by $\pi|_T$.
A classical result of Erd\H{o}s and Szekeres~\cite{erdos1935combinatorial} states that, given two permutations $\pi_1, \pi_2$ of $[t^2 + 1]$, there exists a subset $T \subseteq [t^2 + 1]$, $|T|\ge t+1$, such that both $\pi_1|_T$ and $\pi_2|_T$ are monotone. 
Using the $t=2$ case of the Erd\H{o}s-Szekeres theorem iteratively, Spencer~\cite{1971_Spencer} established the following result.
\begin{lemma}[Spencer, 1971]\label{lem:spencer}
	Let $k$ be a positive integer, $n = 2^{2^{k-1}} + 1$, and $\pi_1, \ldots, \pi_{k+1}$ be $k+1$ permutations of $[n]$. 
	Then, there exist three distinct elements $a,b,c \in [n]$ such that for every $i \in [k+1]$, $\pi_i|_{\{a,b,c\}}$ is monotone.
\end{lemma}

Lemma~\ref{lem:spencer} presents common monotonous triplets in a relatively high number of orders.
These triplets were helpful in \cite{2011_Chandran} for some other bounds as well.
\begin{proof}[of Theorem~\ref{thm:lower_bound_box_line_graph}]
We can suppose $n = 2^{2^{k-1}} + 1$ and prove $\boxi( \overline{L(K_{n})} ) \geq k + 2$, since for all $n$ we can choose the largest $k$ for which $2^{2^{k-1}} + 1 \leq n$, and then it is sufficient to apply the result to $n':=2^{2^{k-1}} + 1$. 

Assume for a contradiction that there exists an interval-order cover $\{C_1, \ldots,$ $C_{k+1}\}$ of $\overline{L(K_{n})}$ consisting of maximal interval-order subgraphs  of $\overline{L(K_{n})}$.
By Lemma~\ref{lem:maximal_interval_order_co_line_graph}, this interval cover can be given with $k+1$ permutations $\pi_1, \ldots, \pi_{k+1}$ of $V(K_{n})$ such that $C_i = \overline{L(K_n)}^{\sigma_i}$, where $\sigma_i \in \Sigma(\pi_i)$.  

Now, by Lemma~\ref{lem:spencer}, there exist three vertices $v_a,v_b,v_c \in V(K_{n})$ such that for every $i \in [k+1]$, the restriction $\pi^i|_{\{v_a,v_b,v_c\}}$ is monotone.
This implies that the edge in $\overline{L(K_n)}$ between $v_av_c\in E(K_n)$ and $v_bv_d \in E(K_n)$, where $v_d$ is any vertex in $V(K_n)$ different from $v_a$,$v_b$ (which exists because $n\ge 5$), is not covered.
\qed
\end{proof}

\section{Conclusion}\label{sec:conclusion}
In this paper, we focused on interval-order subgraphs of line graphs and showed that their number can be bounded by a polynomial of the number of vertices, leading to the polynomial solvability of related optimization problems.
Using the connection between interval orders and boxicity, we developed general arguments to determine the boxicity of complements of line graphs, including the Petersen graph and a polynomial algorithm whenever the boxicity is bounded.
Additionally, we explored extensions of these arguments, addressing interval-order subgraphs of complements of line graphs and providing new bounds on the boxicity of line graphs.

It is also tempting to apply the meta-method  -- of colouring the edges of the graph so that each colour forms an interval-order subgraph, and an edge may get several colours -- to some popular  ``Mycielsky graphs'' and  ``Ramsey graphs'' and follow a program of generalizations similar to what we did in this article. An account of these results can be found in \cite[Chapter 4]{Caoduro_thesis}.

\medskip

We conclude with a few open problems.
The first is the question posed by Stehl{\'\i}k~\cite{QuestionMatej} for \cite{2023_Caoduro} that has eventually led to the present work. 
While this question is resolved in this paper for the case $k=2$, the general case remains open.

\begin{problem}
Determine the boxicity of Kneser graphs $\KG(n,k)$ with $k\geq 3$. 
\end{problem}
Already establishing $\boxi(\KG(7,3))$ is open.
If this value is $3$, we may hope for $\boxi(\KG(n,k)) = \chi(\KG(n,k)) = n - 2k + 2$.
Alternatively, if it is $4$,  $\boxi(\KG(n,k)) = n - k$ may be true. 
Both options are compatible with our result for $k=2$.

Deciding whether the boxicity of the complement of a line graph is at most $k$ can be solved in polynomial time when $k$ is fixed in advance (Theorem~\ref{thm:box_co_line}). We did not manage to show that fixing $k$  is really necessary.

\begin{problem}\label{prob:co_line_NP_hard}
    Is computing the boxicity of \emph{complements of line graphs} $NP$-hard?
\end{problem}

Assuming a positive answer to Problem \ref{prob:co_line_NP_hard}, we ask if the bound as a function of  $k$ in Theorem~\ref{thm:box_co_line} can be essentially improved.
For instance, \emph{is the computation of boxicity $FPT$ for complements of line graphs?}
Problem \ref{prob:co_line_NP_hard} arises for line graphs as well, where even the membership to the class $XP$ is open:

\begin{problem}\label{prob:line_graph}
Can it be decided in polynomial time whether the boxicity of a  \emph{line graph} is $2$?  
\end{problem}

It is easy to check that $\boxi(L(K_4)) = 3$. Hence, Problem \ref{prob:line_graph} is open only for line graphs of graphs with clique number at most three.

Last, as we saw, boxicity is closely related to minimum covers by interval-order graphs.
This can be achieved by covering the complement with large interval-order subgraphs or, equivalently, by intersecting a small number of small interval completions.
The latter motivates the following problem.

\begin{problem}\label{prob:completion}
What is the complexity of the interval completion problem for line graphs?
\end{problem}

\noindent
Recall that this problem is solved in Section~\ref{sec:petersen} for complements of line graphs.

\subsubsection*{Acknowledgement.}
The authors thank Mat\v{e}j Stehl\'\i k for his question about the boxicity of Kneser graphs~\cite{QuestionMatej} that was the starting momentum of this research.
Marco Caoduro was supported by a Natural Sciences and Engineering Research Council of Canada Discovery Grant [RGPIN-2021-02475]. 

\bibliographystyle{splncs04}
\bibliography{references}

\end{document}